\documentclass[12pt]{article}
\usepackage{amsmath,amsthm,amsfonts,amssymb,amscd,cite}
\usepackage[bf,hypcap]{caption}
\usepackage{xcolor}
\usepackage{wrapfig}
\usepackage{graphicx,float}
\usepackage{url,floatflt}
\usepackage{enumerate}
\usepackage{fullpage}
\usepackage{setspace}
\usepackage[colorlinks=true,linkcolor=red,
filecolor=brown,citecolor=blue,urlcolor=blue,pagebackref]{hyperref}

\def\diam{\mathrm{diam}}
\newcommand{\irr}{{\rm irr}}
\newcommand{\imb}{{\rm imb}}
\newcommand{\Var}{{\rm Var}}

\newcommand{\Tr}{{\rm Tr}}

 \newtheorem{theorem}{Theorem}
 \newtheorem{lemma}{Lemma}
 \newtheorem{corollary}{Corollary}
 \newtheorem{proposition}{Proposition}
 
 \newtheorem{remark}{Remark}

\begin{document}

\title{On the transmission-based graph topological indices}

\author{Reza Sharafdini$^{a,\ast}$ \quad Tam\'{a}s R\'{e}ti$^b$\\
$^a$\small{Department of Mathematics, Persian  Gulf University, Bushehr 75169-13817, Iran}\\
$^b$\small{\'{O}buda University B\'{e}csi\'{u}t 96/B, H-1034 Budapest, Hungary}\\
}

\date{29 Aug. 2017}
\maketitle

\begin{abstract}
The distance $d(u,v)$ between the vertices $u$ and $v$ of a connected graph $G$ is defined as the number of edges in a minimal path connecting them.
The \emph{transmission} of a vertex $v$ of  $G$  is defined by $\sigma(v)=\sum\limits_{u\in V(G)}{d(v,u)}$.
In this article we aim to define some transmission-based topological indices.
We obtain lower and upper bounds on these indices and characterize graphs for which these bounds are best possible.
Finally,  we find these indices
for various graphs using the group of automorphisms of $G$. This is an efficient method of finding these indices especially when the automorphism group of $G$ has a few orbits on $V(G)$ or $E(G)$.
\\[2mm]
\end{abstract}

\section{Introduction and Preliminaries}
Let  $G$ be a simple connected graph with  the finite vertex set $V(G)$ and the edge set $E(G)$, and denote by $n=\left| V(G) \right|$ and $m=\left| E(G) \right|$ the number of vertices and edges,  respectively. Using the standard terminology in graph theory, we refer the reader to \cite{west1}. The degree $d(u)$ of the vertex $u\in V(G)$  is the number of the edges incident to $u$.
The edge of the graph  $G$  connecting the vertices $u$ and $v$ is denoted by $uv$.

The role of molecular descriptors (especially topological descriptors) is remarkable in mathematical chemistry especially in QSPR/QSAR
investigations. In mathematical chemistry, the \emph{first Zagreb} index $M_1(G)$ and the \emph{second Zagreb} index $M_2(G)$ belong to the family of the most important degree-based molecular descriptors. They are defined as \cite{Gut-Das-4},\cite{Gut-5},\cite{Ili-Steva-6},\cite{Li-Yo-7},\cite{Nik-Kov-Mil-Tri-3}

\[{{M}_{1}}( G )=\sum\limits_{uv\in E( G )}{d(u)+d(v)}=\sum\limits_{u\in V( G )}{{{d}^{2}}}(u), \qquad {{M}_{2}}( G )=\sum\limits_{uv\in E( G )}{d(u)d(v)}.\]
Similarly, the \emph{first variable Zagreb} index and the \emph{second variable Zagreb} index are defined as \cite{Mil-Nik-VariaZa},\cite{Nik-Kov-Mil-Tri-3},\cite{Xa-Sur-Gut}

\[ M_1^\lambda (G)=\sum\limits_{u\in V(G)} d(u)^{2\lambda},
\quad M_2^\lambda(G)=\sum\limits_{uv\in E(G)} d(u)^\lambda d(v)^\lambda,\]
where $\lambda$ is a real number.

The \emph{Randic} index $R(G)$, the \emph{ordinary sum-connectivity} index $X(G)$, the \emph{harmonic} index $H(G)$ and
\emph{geometric-arithmetic} index $GA(G)$ are also widely used degree-based topological indices \cite{Randic},\cite{Zhang-9},\cite{Fath-Furt-Gut-1},\cite{vuk-Fur-GA},\cite{Zho-Tri-8},\cite{Zho-12}. By definition,
\[R( G )=\sum\limits_{uv\in E(G)}{\frac{1}{\sqrt{d(u)d(v)}}},\quad
X( G )=\sum\limits_{uv\in E(G)}{\frac{1}{\sqrt{d(u)+d(v)}}},\]
\[H( G )=\sum\limits_{uv\in E(G)}{\frac{2}{d(u)+d(v)}}, \quad	
GA(G)=\sum_{uv\in E(G)}\frac{2\sqrt{d(u)d(v)} }{d (u)+d(v)}.\]
 Let  $\Delta= \Delta(G)$ and  $\delta= \delta ( G )$ be the maximum and the minimum degrees, respectively, of vertices of  $G$.  The average degree of  $G$  is $\frac{2m}{n}$.  A connected graph  $G$  is said to be \emph{bidegreed} with degrees  $\Delta$  and  $\delta$  ( $\Delta>\delta \ge 1$), if at least one vertex of  $G$  has degree
$\Delta $ and at least one vertex has degree  $\delta$, and if no vertex of  $G$  has degree different from  $\Delta$  or $\delta$. A connected bidegreed bipartite graph is called \emph{semi-regular} if each vertex in the same part of a bipartition has the same degree.
A graph $G$ is called \emph{regular} if all its vertices have the same degree,
otherwise it is said to be \emph{irregular}. In many applications and problems in theoretical chemistry, it is important to know how  a given graph is irregular.
The (vertex) regularity of a graph is defined in several approaches. Two most frequently used graph topological indices that
measure how irregular a graph is, are the \emph{irregularity} and \emph{variance of degrees}.
Let ${\imb}(e)=\left|d(u)-d(v)\right|$ be the \emph{imbalance} of an edge $e=uv \in E(G)$.
In \cite{AlbertsonIrr}, the \emph{irregularity} of $G$, which is a measure of irregularity of graph $G$, defined as

\begin{equation}\label{eqn:003}
{\irr}(G) = \sum_{e \in E(G)}  {\imb}(e) = \sum_{uv \in E(G)} | d_G(u) - d_G(v) |.
\end{equation}

The \emph{variance of degrees} of graph $G$ is defined as \cite{Bell:1}
\begin{equation}\label{eqn:variancedeg}
 \Var(G)=\frac{1}{n}\sum_{u\in V(G)}\left(d(u)-\frac{2m}{n}\right)^2=\frac{M_1(G)}{n}-\frac{4m^2}{n^2}.
\end{equation}

Another measure of irregularity, which is called \emph{degree deviation}, defined as \cite{Nikiforov-degdevi}
\[s(G)=\sum_{u\in V(G)}\Big|d(u)-\frac{2m}{n}\Big|.\]
It is worth mentioning that $\frac{s(G)}{n}$ is noting but the \emph{mean deviation} of the data set
$\left\{d(u)\mid u\in V(G)\right\}$.

The \emph{distance} between the vertices $u$ and $v$ in graph $G$ is denoted by $d(u,v)$ and it is defined as the number of edges in a minimal path connecting them. The \emph{eccentricity} $\varepsilon(v)$ of a vertex $ v $ is the maximum distance from $ v $ to any other vertex. The \emph{diameter} $ \diam(G) $ of $G$ is the maximum eccentricity among the vertices of $G$.
The \emph{transmission} (or \emph{status}) of a vertex $v$ of  $G$  is defined as $\sigma(v)=\sigma_G(v)=\sum\limits_{u\in V(G)}{d(v,u)}$.
A graph  $G$  is said to be \emph{transmission regular} \cite{Aou-Han-3} if  $\sigma(u)=\sigma(v)$  for any vertex $u$ and $v$ of $G$.
A transmission regular graph  $G$  is called \emph{$k$-transmission regular} if there exists a positive integer $k$, for which  $\sigma(v)=k$  for any vertex $v$ of  $G$. In $K_n$, the complete graph of order $n$, each vertex has transmission $n-1$. So it is $(n-1)$-transmission regular.  The
the cycle $C_n$ and the complete bipartite graph $K_{a,a}$ are transmission regular. It has been verified that there exist regular and non-regular  transmission regular graphs \cite{Aou-Han-3}.
Consider the polyhedron depicted in Figure \ref{Rhombicdod}. It is the rhombic dodecahedron that contains 14 vertices, (8 vertices of degree 3 and 6 vertices of degree 4), 24 edges and 12 faces, all of them are congruent rhombi.
\begin{figure}[H]
  \centering
  \includegraphics[width=5cm]{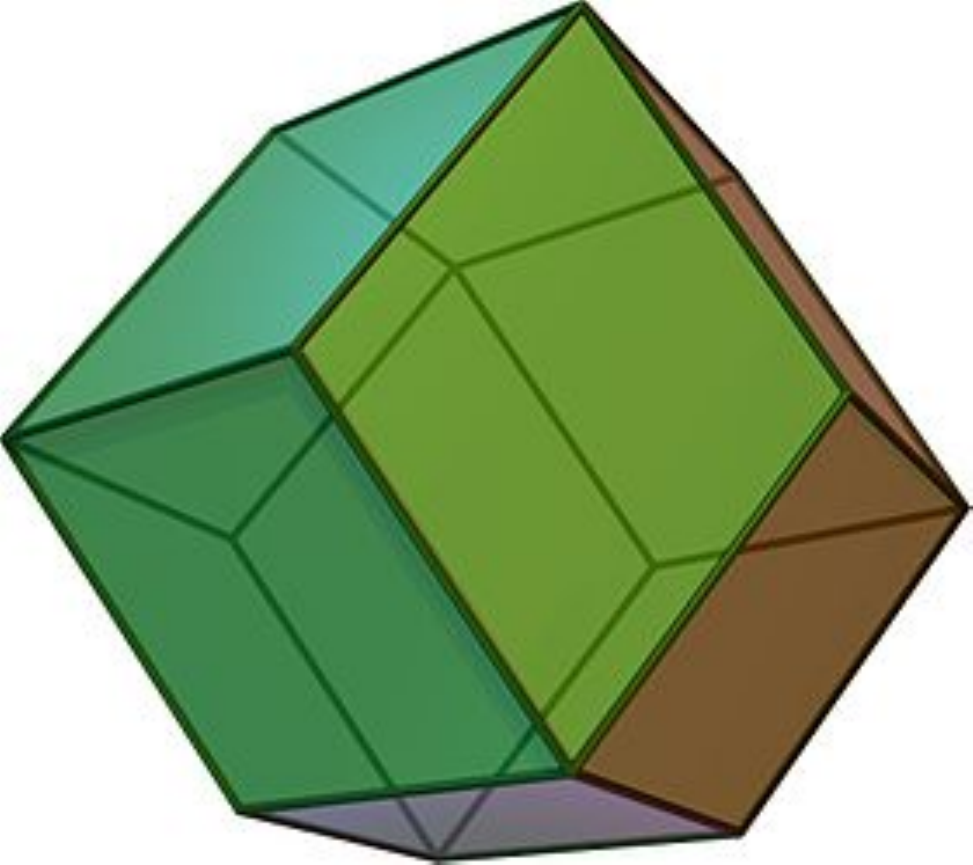}
  \caption{The rhombic dodecahedron}\label{Rhombicdod}
\end{figure}
The graph $G_{RD}$ of the rhombic dodecahedron is a bidegreed, semi-regular 28-transmission regular graph (See Figure \ref{GRD}).
An interesting observation is that the 14-vertex polyhedral graph $G_{RD}$ depicted in Figure \ref{GRD} is identical to the semi-regular graph published earlier in an alternative form in \cite{Aou-Han-3}.
It is conjectured that $G_{RD}$ is the smallest non-regular, bipartite, polyhedral (3-connected) and transmission regular graph.

\begin{figure}[H]
  \centering
  \includegraphics[width=5cm]{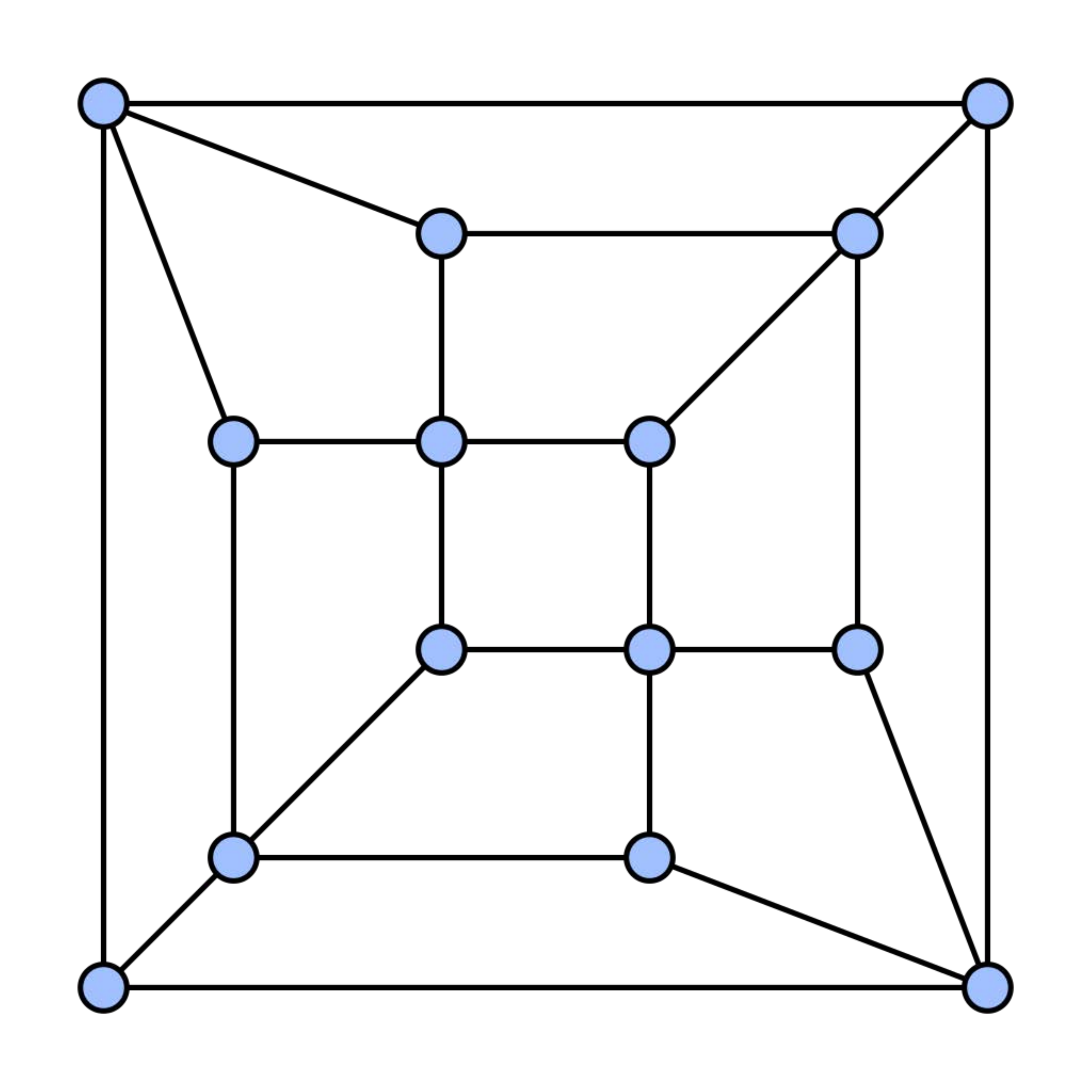}
  \caption{The rhombic dodecahedron witch is 28-transmission regular graph but
not regular}\label{GRD}
\end{figure}

If $\omega$ is a vertex weight of graph $G$, then one can see that
\begin{equation}\label{eq:weightdistance}
\sum_{\{u,v\}\subseteq V(G)}\left(\omega(u)+ \omega(v)\right)d(u, v) =\sum_{v\in V (G)}
\omega(v)\sigma (v).	
\end{equation}

It is easy to construct various transmission-based indices having the same structure as the known degree-based topological indices. Based on this analogy-concept, the corresponding transmission-based indices are defined.

Let us define the \emph{transmission Randi\'c} index $RS(G)$, the \emph{transmission ordinary sum-connectivity} index $XS(G)$, the \emph{transmission harmonic} index $HS(G)$ and
the \emph{transmission geometric-arithmetic} index $GAS(G)$ as follows:
\[RS( G )=\sum\limits_{uv\in E(G)}{\frac{1}{\sqrt{\sigma(u)\sigma(v)}}},\quad
XS( G )=\sum\limits_{uv\in E(G)}{\frac{1}{\sqrt{\sigma(u)+\sigma(v)}}},\]
\[HS( G )=\sum\limits_{uv\in E(G)}{\frac{2}{\sigma(u)+ \sigma(v)}}, \quad	
GAS(G)=\frac{n}{2m}\sum_{uv\in E(G)}\frac{2\sqrt{\sigma(u)\sigma(v)} }{\sigma (u)+\sigma(v)}  .\]
It follows that $GAS(G)\leq \frac{n}{2}$,  with equality if and only if $G$ is a transmission regular graph.


The \emph{Wiener} index $W( G )$, the {\em Balaban} index $J(G)$  and the
{\em sum-Balaban} index $SJ(G)$ represent a particular class of transmission-based topological indices. They are defined as
\cite{B1},\cite{B2},\cite{Bsum},\cite{deng-sum},\cite{Das-Gut-15},\cite{Hosoya},\cite{Ent-Jak-Sny-12},\cite{Gut-Yeh-14}

\[W(G )=\frac{1}{2}\sum\limits_{u\in V G )}{\sum\limits_{v\in V( G )}{d(u,v)}}=\frac{1}{2}\sum\limits_{u\in V G )}{\sigma (u)},
\]		

$$
J(G)=\frac m{m-n+2}\sum_{uv\in E(G)}\frac 1{\sqrt{\sigma(u)\sigma(v)}}=
\frac m{m-n+2}RS(G),
$$

$$
SJ(G)=\frac m{m-n+2}\sum_{uv\in E(G)}\frac 1{\sqrt{\sigma(u)+\sigma(v)}}=\frac m{m-n+2}XS(G).
$$
In \cite{RamYalJAMC} the \emph{first transmission Zagreb} index $MS_1(G)$ and  the \emph{second transmission Zagreb} index $MS_2(G)$ are defined as

\[MS_{1} (G)=\sum_{uv\in E(G)}\sigma (u)+\sigma(v)=\sum_{u\in V(G)}d(u)\sigma (u), \qquad
MS_{2} (G)=\sum_{uv\in E(G)}\sigma (u)\sigma (v).\]
It is important to note that $MS_{1}(G)$  coincides with the \emph{degree distance} $DD(G)$ that was introduced in \cite{Dobr-Koch-DD}, \cite{Gut--Schultz94} and \cite{Tom-16}
In fact by Eq. \eqref{eq:weightdistance},

\begin{equation}\label{eq:degreedistance}
DD(G)=\sum_{\{u,v\}\subseteq V(G)}(d(u)+ d(v))d(u, v) =\sum_{v\in V (G)}
d(v)\sigma (v)=MS_{1}(G).	
\end{equation}
Consequently, if $G$ is a $k$-transmission regular graph with $m$ vertices, then $DD(G)=MS_1(G)=2mk$.

Let us propose the \emph{variable degree transmission Zagreb} index
 $ MSD^\lambda(G)$ and the \emph{variable transmission Zagreb} index
$ MS^\lambda(G)$ as follows

\[ MSD^\lambda(G)= \sum_{u\in V(G)}d(u) \sigma(u)^{2\lambda-1},\quad MS^\lambda(G)= \sum_{u\in V(G)} \sigma(u)^{2\lambda},\]
where $\lambda$ is a real number.

The \emph{eccentric distance sum} of a graph $G$, denoted by $\xi^d(G)$, defined as \cite{Gupta-Sin-MadaEDS}
\[
\xi^d(G)=\sum_{u\in V (G)}
\varepsilon(u)\sigma (u).
\]
It follows from   Eq. \eqref {eq:weightdistance} that
\begin{equation}\label{eq:EDS-expand}
\xi^d(G)=\sum_{\{u,v\}\subseteq V(G)}(\varepsilon(u)+ \varepsilon(v))d(u, v) =\sum_{v\in V (G)}
\varepsilon(v)\sigma (v).	
\end{equation}

Inspired from   Eq. \eqref {eqn:003} and   Eq. \eqref {eqn:variancedeg} we define two transmission-based irregularity as follows:
Let $G$ be a connected graph with $n$ vertices and $m$ edges.
The \emph{transmission imbalance} of an edge $e=uv \in E(G)$ is defined as  ${\imb_\Tr}(e)=\left| \sigma_G(u)- \sigma_G(v)\right|$.
Let us define the \emph{transmission irregularity} ${\irr_{\Tr}}(G)$  and the \emph{transmission variance} $\Var_{\Tr}(G)$ of $G$ as follows:

\begin{equation}\label{eqn:003}
{\irr_{\Tr}}(G) = \sum_{e \in E(G)}  {\imb_\Tr}(e) = \sum_{uv \in E(G)} |  \sigma_G(u) -  \sigma_G(v) |.
\end{equation}

\begin{equation}\label{eqn:variancedeg}
 \Var_{\Tr}(G)=\frac{1}{n}\sum_{u\in V(G)}\Big( \sigma_G(u)-\dfrac{2W(G)}{n}\Big)^2=\frac{ MSD^1(G)}{n}
-\frac{4W(G)^2}{n^2}\ge 0.
\end{equation}
Note that $\dfrac{2W(G)}{n}$ is nothing but the vertex transmission average of graph $G$. It is obvious that $ \Var_{\Tr}(G)$ is equal to zero if and only if $G$ is transmission regular.

Let us also define the transmission-based topological indices $QS_e(G)$ and $QS_{v,e}(G)$ as follows
\begin{align*}
QS_{e} (G)=\frac{1}{m}  {\irr_{\Tr}}(G), \qquad
QS_{v,e} (G)&=\frac{n}{2} \left\{1+\frac{1}{m}  {\irr_{\Tr}}(G) \right\}=\frac{n}{2} \left\{1+QS_{e} (G)\right\}.
\end{align*}

\begin{remark}
Let $G$ be an $n$-vertex graph. Comparing topological indices $GAS(G)$ and $QS_{v,e}(G)$, we get
$$GAS(G)\le \frac{n}{2}\le QS_{v,e} (G).$$
Equalities hold in both sides simultaneously if and only if $G$ is  transmission regular.
\end{remark}

\section{Establishing lower and upper bounds}

\begin{lemma}\label{lem:upper1}
Let $G$ be a connected graph with $n\ge 2$ vertices and $m$ edges. Then

\[0\le  {\irr_{\Tr}}(G)  \le m(n-2) ,\]

\[0\le \sum_{uv\in E(G)}\left(\sigma (u)-\sigma (v)\right)^{2} \le m(n-2)^{2}.\]
The equality on the right-hand sides holds if and only if  $G$ is isomorphic to $S_n$. The equality on the left-hand sides holds if and only if  $G$ is
transmission regular.

\end{lemma}
\begin{proof}
For an arbitrary edge $uv$  of $G$,  we have
$\left|\sigma (u)-\sigma (v)\right|\le n-2.$ Therefore,
\begin{align*}
\irr_{\Tr}(G)= \sum_{uv \in E(G)}  | \sigma(u) -  \sigma(v) | \leq \sum_{uv \in E(G)} (n-2) = m(n-2).
\end{align*}
  It is trivial that in both formulas the equality on the right-hand side holds if and only if
  $G$ isomorphic to $S_n$, since the star is the only graph where equality holds for each edge.
\end{proof}

\begin{corollary}
Let $T$ be a tree with $n\ge 2$ vertices. Then

\[0\le  {\irr_{\Tr}}(T)  \le (n-1)(n-2) ,\]

\[0\le \sum_{uv\in E(T)}\left(\sigma (u)-\sigma (v)\right)^{2} \le (n-1)(n-2)^{2}.\]
The equality on the right-hand sides holds if and only if  $G$ is isomorphic to $S_n$. The equality on the left-hand sides holds if and only if  $G$ is
transmission regular.
\end{corollary}
\begin{proof}
  It is a consequence of Lemma \ref{lem:upper1} and the fact that a tree with $n$ vertices has exactly $n-1$ edges.
\end{proof}

\begin{corollary}
Let $G$ be a connected graph with $n\ge 2$ vertices. Then

\[(n-2)\ge QS_{e} (G)\ge 0\]
and
\[\frac{n(n-1)}{2} \ge QS_{v,e} (G)\ge \frac{n}{2}.\]
The upper bounds are achieved if and only if $G$ is isomorphic to $S_n$  and the lower bounds are achieved if and only if $G$ is transmission regular.
\end{corollary}
\begin{proof}
  It is a direct consequence of Lemma \ref{lem:upper1}.
\end{proof}

\begin{lemma}\label{lem:diam2Rama}

Let $G$ be a connected graph with $n\ge 3$ vertices and with maximum vertex degree $\Delta$. Then for each arbitrary vertex $u$ of $G$
$$\sigma (u) \geqslant 2(n - 1) - d(u) \geqslant 2(n - 1) - \Delta  \geqslant n - 1.$$
\end{lemma}

\begin{proof}
  Because $n-1\ge \Delta\ge d(u)$ one obtains that
\begin{align*}\label{}
\sigma (u) &= \sum\limits_{\left\{ {w \in V \mid d(u,w) = 1} \right\}} {d(u,w) + \sum\limits_{\left\{ {w \in V \mid d(u,w) > 1} \right\}} {d(u,w)} = d(u) + } \sum\limits_{\left\{ {w \in V \mid d(u,w) > 1} \right\}} {d(u,w)}  \\
&\geqslant
d(u) + 2(n - 1 - d(u)) = 2n - 2 - d(u) \geqslant 2(n - 1) - \Delta  \geqslant n - 1.
\end{align*}
\end{proof}

\begin{remark}
There are several graphs containing a vertex $u$ for which $\sigma(u) = n-1$. For example, $\sigma(u)=d(u)=n-1$ for any vertex $u$ of a complete graph $K_n$.
\end{remark}
\begin{remark}\label{remark:eccent}
Let $G$ be a connected graph.
It is easy to see that for any $u\in V(G)$, $\sigma (u) \geqslant 2(n - 1) - d(u)$, with equality if and only if  $\varepsilon(u)\leq 2$. This implies that
\begin{enumerate}
  \item[{(\rm i)}]  $\sigma (u)=2(n-1)-d(u)$ for any vertex
$u$ of a connected graph $G$ if and only if $\diam(G)\leq 2$.
  \item[{(\rm ii)}] Let $G$ be a connected graph with  $\diam(G)\leq 2$. Then $G$ is transmission regular if and only if  $G$  is regular.
\end{enumerate}

\end{remark}

\begin{proposition}
Let $G$ be a connected graph with $n$ vertices. Then
\[ MSD^{\frac{3}{2}}(G)  \ge 2(n-1)MS^1(G) -MS^\frac{3}{2}(G),\]
with equality if and only if  $\diam(G)\leq 2$.
\end{proposition}

\begin{proof}
It follows from Lemma \ref{lem:diam2Rama} that
\[\sum_{ u\in V(G) }d(u)\sigma^{2} (u) \ge \sum_{ u\in V(G) }\left(2n-2-\sigma (u)\right) \sigma^{2} (u)=2(n-1)\sum_{ u\in V(G) }\sigma^{2} (u) -\sum_{ u\in V(G) }\sigma^{3} (u) ,\]
and by Remark \ref{lem:diam2Rama}, the equality holds if and only if  $\diam(G)\leq 2$.
\end{proof}

\begin{proposition}\label{prop:MS1lower}
Let $G$ be a connected graph with $n$ vertices. Then
\[MS_{1} (G)\ge 4(n-1)W(G)-MS^1(G),\]
with equality if and only if  $\diam(G)\leq 2$.
\end{proposition}
\begin{proof}
It follows from Lemma \ref{lem:diam2Rama} that
\[\sum_{ u\in V(G) }d(u)\sigma (u) \ge \sum_{ u\in V(G) }\left(2n-2-\sigma (u)\right) \sigma (u)=2(n-1)\sum_{ u\in V(G) }\sigma (u) -\sum_{ u\in V(G) }\sigma^{2} (u) .\]
It follows from Remark \ref{remark:eccent} that the equality holds if and only if $\diam(G)\leq 2$.
\end{proof}

\begin{lemma}\label{lem:diam2}
Let $G$ be a connected graph with $n$ vertices and $m$ edges. If $\diam(G)\leq 2$, then
\begin{enumerate}
  \item[{\rm(i)}]
      ${\irr_{\Tr}}(G)  =\irr(G) \ge 0 .$
  \item[{\rm(ii)}]
       $QS_{v,e} (G)=\frac{n}{2} \left\{1+\frac{1}{m} \irr(G) \right\}\ge \frac{n}{2}.$
\end{enumerate}
In particular, in both cases equality holds if and only if $G$ is regular.
\end{lemma}
\begin{proof}
(i) It is a direct consequence of Lemma \ref{lem:upper1} and Remark \ref{remark:eccent}.
(ii) It follows directly from part (i).
\end{proof}

\begin{corollary}

Let $K_{p,q}$ be the complete bipartite graph with $p+q$ vertices and with parts of size $p$ and $q$. Then
\begin{enumerate}
  \item[{\rm(i)}]
  ${\irr_{\Tr}} (K_{p,q} )=pq\left|p-q\right|\ge 0 .$
  \item[{\rm(ii)}]
$
QS_{v,e} (K_{p,q} )=\frac{p+q}{2} \left\{1+\left|p-q\right|\right\}\ge \frac{p+q}{2},
$
Specially
$QS_{v,e} (S_{n} )=\frac{n(n-1)}{2}.$
\end{enumerate}
In particular, the equalities in {\rm (i)} and {\rm (ii)}  hold if and only if $p=q$.
\end{corollary}

\begin{proof}
  (i) Since $\diam(K_{p,q})=2$ and $|E(K_{p,q})|=pq$, it follows from Lemma \ref{lem:diam2} (i) that
$
{\irr_{\Tr}}(K_{p,q})=\irr(K_{p,q})=\sum_{uv\in E(K_{p,q})}\left|p-q\right|=pq\left|p-q\right|.
$
(i) Since $\diam(K_{p,q})=2$ and $|V(K_{p,q})|=p+q$,
it follows from Lemma \ref{lem:diam2} (ii) that
\[
QS_{v,e} (K_{p,q} )=\frac{p+q}{2}\left\{1+\left|p-q\right|\right\}\ge \frac{p+q}{2}.
\]
Specially, let $n\ge 2$ and $p=1$ and $q=n-1$. Then $K_{p,q}$ is isomorphic to the star $S_n, (n=p+q)$. Consequently, we obtain that
\[QS_{v,e} (S_{n} )=\frac{n}{2}(1+|2-n|)=\frac{n(n-1)}{2} .\]
It follows from Lemma \ref{lem:diam2} that the equalities in {\rm (i)} and {\rm (ii)}  hold if and only if $K_{p,q}$ is regular if and only if $p=q$.
\end{proof}

An edge $uv$ of a connected graph $G$ is said to be a \textit{strong edge} of $G$,  if $\left|d(u)-d(v)\right|>0$. Denote by $es(G)$ the number of strong edges of $G$. It is obvious that if $G$ is a connected graphs, then $es(G)=0$ if and only if $G$ is regular. From this observation it follows that the topological invariant $es(G)$ can be considered as a graph irregularity index. There are several graphs in which each edge is strong, that is $es(G)=|E(G)|$. For example,
$es(K_{p,q})=|E(K_{p,q})|=pq$  if $p$ is not equal to $q$. It can be easily constructed a tree graph $T$ with an arbitrary large edge number $m(T)$, for which $es(T)=m(T)$. Consider the $(n\ge 5)$-vertex windmill graph denoted by $Wd(n)$. It is a graph with diameter 2, with the vertex number $n=2k+1$ and with the edge number $m=3k$, where $k\ge 2$ is an arbitrary positive integer.
Note that $es(Wd(n))=2k=\frac{2}{3}m=n-1$.

\begin{proposition}\label{prop:Wd}
For the windmill graph $Wd(n)$ we have
\begin{enumerate}
  \item[{\rm (i)}]
${\irr_{\Tr}}(Wd(n))=es(Wd(n))(n-3)=\frac{2}{3} m(n-3)=(n-1)(n-3).$
  \item[{\rm (ii)}]
  $QS_{v,e} (Wd(n))=\frac{n}{2} \left\{1+\frac{2}{3} (n-3)\right\}.$
\end{enumerate}
\end{proposition}

\begin{proof}
(i)
Let $E_0$ be the set of strong edges of $Wd(n)$. It is easy to see that
$$E_0=\big\{uv\in E(Wd(n))\mid d(u)=2, d(v)=n-1\big\}, \quad es(Wd(n))=|E_0| .$$
Since $\diam(Wd(n))=2$, it follows from Lemma \ref{lem:diam2} (i) that
\begin{align*}
{\irr_{\Tr}}(Wd(n))=\irr(Wd(n))
&=\sum_{uv\in E_0}\left|d(u)-d(v)\right|
=\sum_{uv\in E_0}\left|2-(n-1)\right|\\
&=es(Wd(n))\left|2-(n-1)\right|
\\&=\frac{2}{3} m(n-3)=(n-1)(n-3).
\end{align*}
(ii)
It follows from part (i) that
\begin{align*}
QS_{v,e} (Wd(n))
&=\frac{n}{2} \left\{1+\frac{1}{m} \irr_\Tr (Wd(n))\right\}=\frac{n}{2} \left\{1+\frac{2}{3} (n-3)\right\}\\
&=\frac{n}{2} \left\{1+\frac{1}{m} (n-1)(n-3)\right\}.
\end{align*}
\end{proof}

\begin{lemma}[\cite{Liu-Pan}]\label{lem:trPn} Let
$P_n$ be a path of order $n$, and let $V(P_n)=\{v_0, v_1, \ldots ,v_{n-1}\}$ such that $E(P_n)=\{v_iv_{i+1}| i=0, \ldots, n-2 \}$. Then for $0\leq i \leq n-1$
\[\sigma_{P_n} ( v_i)=\frac{1}{2}\Big(2i^2 -2( n-1 ) i +( n-1 )^2 +( n-1 )\Big).\]
\end{lemma}

The following is a direct consequence of Lemma \ref{lem:trPn}.

\begin{proposition}\label{prop:IrrTrPn}
The transmission irregularity index of $P_n$ is given by
\[
 \irr_{\Tr}(P_n)=
\begin{cases}
  \frac{n(n-1)}{2},  & \mbox{if $n$ is even}, \\[3mm]
  \frac{(n-1)^2}{2}, & \mbox{if $n$ is odd}.
\end{cases}
\]
\end{proposition}



For an edge $uv$  of a connected graph $G$,  define the positive integers $N_u$ and $N_v$ where $N_u$ is the number of vertices of $G$ whose distance to vertex $u$   is smaller than distance to vertex $v$, and analogously, $N_v$ is the number of vertices of $G$ whose distance to the vertex $v$ is smaller than to $u$. The number of vertices equidistant from $u$  and $v$ is denoted by $N_{uv}$.  An edge $uv$  of $G$ is called a distance-balanced edge if $N_u=N_v$. A graph $G$ is said to be \textit{distance-balanced} \cite{Il-KlMil-2} if its each edge is distance-balanced.  It is known that a connected graph $G$ is transmission regular if and only if $G$ is distance balanced \cite{Aou-Han-3},\cite{Il-KlMil-2}.

The \emph{Szeged} index $Sz(G)$ and the \emph{revised Szeged} index $Sz^{*}(G)$ of a connected graph $G$ are defined as \cite{Kla-Raj-Gut-4},\cite{Na-Kho-Ash-6},\cite{Pi-Ra-5}
\[Sz(G)=\sum_{ uv\in E(G) }N_{u} N_{v}, \qquad Sz^{*} (G)=\sum_{ uv\in E(G) }\left\{N_{u} +\frac{N_{uv}}{2}\right\} \left\{N_{v} +\frac{N_{uv}}{2}\right\}.\]
\begin{remark}\label{rem:SzProperty}
For any connected graph $G$ with $n$ vertices, the following known relations are fulfilled
\cite{Aou-Han-3},\cite{Dob-Gut-11},\cite{Dob-Ent-Gut-13},\cite{Ent-Jak-Sny-12},\cite{Khal-You-Ash-7},\cite{Kla-Raj-Gut-4},\cite{Na-Kho-Ash-6},\cite{Pi-Ra-5},\cite{Xin_Zho-8}

\begin{enumerate}
\item[{\rm (i) }]  For any edge $uv$  of $G$, $n=N_u+N_v+N_{uv}$.  This implies that a graph $G$ is bipartite if and only if $n=N_u+N_v$ holds for any edge $uv$  of $G$;

\item[{\rm(ii)}]  The inequality $Sz(G)\ge W(G)$ is fulfilled;

\item[{\rm(iii)}]  $Sz(G)\leq Sz^{*} (G)$ with equality if and only if $G$ is bipartite;

\item[{\rm(iv)}]  For an $n$-vertex tree $T$, $W(S_n)\leq W(T)\leq W(P_n)$;

\item[{\rm(v)}]  For a tree graph $T$, $Sz^{*}(T)=Sz(T)=W(T)$.
\end{enumerate}
\end{remark}
The fundamental properties of Wiener index and their extremal graphs are summarized in \cite{Das-Gut-15},\cite{Dob-Gut-11},\cite{Ent-Jak-Sny-12},\cite{Dob-Ent-Gut-13},\cite{Gut-Yeh-14}.  Transmission regular  graphs are characterized by the following property:

\begin{lemma}[\cite{Aou-Han-3},\cite{Il-KlMil-2},\cite{Kla-Raj-Gut-4}]
Let $G$ be a connected graph with $n$ vertices and $m$ edges. Then
\[Sz^{*} (G)\le \frac{n^{2} m}{4} ,  \]
with equality if and only if $G$ is  transmission regular.
\end{lemma}

\begin{lemma}[\cite{Aou-Han-3},\cite{Dob-Gut-11}]\label{lem:SN}
Let $G$ be a connected graph and let $uv$  be an edge of $G$. Then
\[\sigma(u)-\sigma (v)=N_v-N_u.\]
\end{lemma}


\begin{lemma}\label{lem:corSN}
Let $G$ be a connected graph. Then the following hold:
\begin{enumerate}
\item[{\rm (i) }]
\[\irr_\Tr(G) =\sum_{ uv\in E(G) }\left|N_{u} -N_{v} \right| \ge 0;\]

\item[{\rm (ii) }]
\begin{align*}
\sum_{ uv\in E(G) }\left(N_{u} -N_{v} \right)^{2} &=  MSD^{\frac{3}{2}}(G) -2MS_{2} (G)\ge 0
\end{align*}

\item[{\rm (iii) }]
\begin{align*}
\sum_{ uv\in E(G) }\left(\sigma (u)-\sigma (v)\right)^{2}&=\sum_{ uv\in E(G) }\left(N_{u}^{2} +N_{v}^{2} \right) -2Sz(G)\ge 0;
\end{align*}

\item[{\rm (iv) }]
\[\sum_{ uv\in E(G) }\left(N_{u}^{2}+N_{v}^{2} \right)= MSD^{\frac{3}{2}}(G) +2Sz(G)-2MS_{2} (G).\]
\end{enumerate}
In {\rm (i)}, {\rm (ii)} and {\rm (iii)} the equality holds if and only if $G$ is transmission regular.
\end{lemma}

\begin{proof}
(i) is a direct consequence of Lemma \ref{lem:SN}.

(ii)
\begin{align*}
0\leq \sum_{ uv\in E(G) }\left(N_{u} -N_{v} \right)^{2} &=\sum_{ uv\in E(G) }\left(\sigma (u)-\sigma (v)\right)^{2}\\  &=\sum_{ uv\in E(G) }\left(\sigma^{2} (u)+\sigma^{2} (v)\right)-2\sum_{ uv\in E(G) }\sigma(u)\sigma(v)\\
&= \sum_{u\in V(G)}d(u)\sigma^{2} (u)-2MS_{2} (G)\\
&=MSD^{\frac{3}{2}}(G)-2MS_{2} (G).
\end{align*}

(iii)
\begin{align*}
0\leq \sum_{ uv\in E(G) }\left(\sigma (u)-\sigma (v)\right)^{2}&=\sum_{ uv\in E(G) }\left(N_{u} -N_{v} \right)^{2}\\
&=\sum_{ uv\in E(G) }\left(N_{u}^{2} +N_{v}^{2} \right) -2Sz(G).\\
\end{align*}
(iv) It follows from the proof of part (ii) and (iii) that

\begin{align*}\label{}
\sum_{ uv\in E(G) }\left(N_{u}^{2} +N_{v}^{2} \right)&=\sum_{ uv\in E(G) }\left(\sigma (u)-\sigma (v)\right)^{2} +2Sz(G)\\
&=MSD^{\frac{3}{2}}(G)-2MS_{2} (G)+2Sz(G).
\end{align*}

\end{proof}

\begin{remark}
Based on Lemma \ref{lem:corSN}, the transmission-based topological index $QS_{v,e}(G)$ can be represented in the following alternative form:
\[QS_{v,e} (G)=\frac{n}{2} \left\{1+\frac{1}{m} \sum_{ uv\in E(G) }\left|\sigma (u)-\sigma (v)\right| \right\}=\frac{n}{2} \left\{1+\frac{1}{m} \sum_{ uv\in E(G) }\left|N_{u} -N_{v} \right| \right\}.\]
\end{remark}

\begin{proposition}\label{prop:SZbi}
Let $G$ be a connected graph with $n$ vertices and $m$ edges. Then
\[n^{2} m\ge  MSD^{\frac{3}{2}}(G) +4Sz(G)-2MS_{2} (G),\]
with equality if and only if $G$ is a bipartite  graph.
\end{proposition}
\begin{proof}
Let  $G$ be a connected graph with $n$ vertices.
It follows from Remark \ref{rem:SzProperty} (i) that for any edge $uv$ of $G$,
$N_u + N_v \leq n$, with equality if and only if $G$ is bipartite. This implies that
\[n^{2} \ge \left(N_{u} +N_{v} \right)^{2} =\left(N_{u}^{2} +N_{v}^{2} \right)+2N_{u} N_{v},\]
with equality if and only if $G$ is bipartite.
Consequently, by Lemma \ref{lem:corSN} (iv) we have

\begin{align*}\label{eq:sn}
n^{2} m = \sum_{ uv\in E(G) }n^{2} &\ge  \sum_{ uv\in E(G) }\left(N_{u}^{2} +N_{v}^{2} \right)+2\sum_{ uv\in E(G) }N_{u} N_{v}\\
&=\sum_{ uv\in E(G) }\left(N_{u}^{2} +N_{v}^{2} \right)+2Sz(G)\\
&=MSD^{\frac{3}{2}}(G)+4Sz(G)-2MS_{2}(G),
\end{align*}
with equality if and only if $G$ is bipartite.
\end{proof}

\begin{proposition}\label{prop:IrrNN}
Let $G$ be a connected graph with $n$ vertices. Then
\begin{equation}\label{eq:T-bipartite}
 {\irr_{\Tr}}(G)  =\sum_{ uv\in E(G) }\left|N_{u}-N_{v} \right|\ge \frac{1}{n}\sum_{ uv\in E(G) }\left|N_{u}^{2} -N_{v}^{2} \right|,
\end{equation}
with equality if and only if $G$ is a bipartite graph.
\end{proposition}

\begin{proof}
Let  $G$ be a connected graph with $n$ vertices.
It follows from Remark \ref{rem:SzProperty} (i) that for any edge $uv$ of $G$,
$N_u + N_v \leq n$, with equality if and only if $G$ is bipartite.
Therefore, it follows from Lemma \ref{lem:SN} and
\[\left|N_{u}^{2}-N_{v}^{2} \right|=(N_{u}+N_{v})\left|N_{u}-N_{v} \right|
\le n\left|N_{u}-N_{v} \right|=n\left|\sigma (u)-\sigma (v)\right|,\]
with equality if and only if $G$ is bipartite.
This implies that Eq. \eqref {eq:T-bipartite} holds with equality if and only if $G$ is bipartite.
\end{proof}

\begin{corollary}
Let $T_n$ be an $n$ vertex tree. Then
\[MS_{2}(T_n)=2W(T_n)+\frac{1}{2}MSD^{\frac{3}{2}}(T_n)-\frac{n^{2}(n-1)}{2},\]
\[  \irr_\Tr  (T_n)=\frac{1}{n} \sum_{ uv\in E(T_n) }\left|N_{u}^{2}-N_{v}^{2} \right|
.\]

\end{corollary}
\begin{proof}
  It is a consequence of Proposition \ref{prop:SZbi}, Proposition \ref{prop:IrrNN} and Remark \ref{rem:SzProperty}, since a tree with $n$ vertices is bipartite and has exactly $n-1$ edges.
\end{proof}

\begin{proposition}[\cite{Dob-Gut-11}]\label{propo:Dob-Gut-11}
Let $G_B$ be a connected bipartite graph with $n$ vertices and $m$ edges. Then
\[Sz^{*}(G_{B} )=Sz(G_{B} )=\frac{n^{2} m}{4} -\frac{1}{4} \sum_{ uv\in E(G_B) }\left(\sigma (u)-\sigma (v)\right)^{2}  \le \frac{n^{2} m}{4}, \]
with equality if and only if $G$ is  transmission regular.
\end{proposition}

\begin{corollary}
Let $G_B$ be a connected bipartite graph with $n$ vertices and $m$ edges. Then
\[QS_{v,e} (G_{B} )\le \sqrt{n^{2} -\frac{4}{m} Sz(G_{B} )}, \]
 with equality if and only if $\left|\sigma (u)-\sigma (v)\right|$ is constant for any edge $uv\in G_B$.
\end{corollary}
\begin{proof}
Using Cauchy-Schwartz inequality and Proposition \ref{propo:Dob-Gut-11} one obtains for $G_B$ that
\[\left\{\frac{1}{m} \sum_{ uv\in E(G_B) }\left|\sigma (u)-\sigma (v)\right| \right\}^{2} \le \frac{1}{m} \sum_{ uv\in E(G_B) }\left(\sigma (u)-\sigma (v)\right)^{2}  =n^{2} -\frac{4}{m} Sz(G_{B} ),\]
with equality if and only if $\left|\sigma (u)-\sigma (v)\right|$ is constant for any edge $uv\in G_B$.
Consequently,
\[\frac{1}{m} \sum_{ uv\in E(G_B) }\left|\sigma (u)-\sigma (v)\right| \le \sqrt{n^{2} -\frac{4}{m} Sz(G_{B} )},\]
with equality if and only if $\left|\sigma (u)-\sigma (v)\right|$ is constant for any edge $uv\in G_B$.
Because
\[QS_{v,e} (G_{B} )-\frac{n}{2} =\frac{n}{2m} \sum_{ uv\in E(G_B) }\left|\sigma (u)-\sigma (v)\right|, \]
we have
\[QS_{v,e} (G_{B} )-\frac{n}{2} \le \frac{n}{2} \sqrt{n^{2} -\frac{4}{m} Sz(G_{B} )}, \]
with equality if and only if $\left|\sigma (u)-\sigma (v)\right|$ is constant for any edge $uv\in G_B$.
\end{proof}

\begin{lemma}[\cite{Dob-Gut-11}]\label{lem:SzTMS1}
Let $T_n$ be an $n$-vertex tree. Then
\[Sz(T_{n} )=W(T_{n} )=\frac{1}{4}\left(n(n-1)+MS_1(T_n) \right).\]
\end{lemma}
The following proposition demonstrates
that the Wiener index and the first transmission Zagreb index are closely related.
\begin{proposition}\label{prop:MS1WSz}
Let $T_n$ be an $n$-vertex tree. Then
\begin{equation}\label{MS1WSz}
MS_{1} (T_{n} )=4W(T_{n})-n(n-1)=4Sz(T_{n})-n(n-1).
\end{equation}
\end{proposition}

\begin{proof}
For any connected graph $G$ we have
\begin{equation*}
MS_{1} (G)=\sum_{ uv\in E(G) }\left(\sigma (u)+\sigma (v)\right) =\sum_{ u\in V(G) }d(u)\sigma (u).
\end{equation*}
Therefore, by Lemma \ref{lem:SzTMS1} the result follows.
\end{proof}

\begin{remark}
As a consequence of Eq. \eqref {MS1WSz}, we conclude that in the family of $n$-vertex trees there is a linear correspondence (a perfect linear correlation) between the topological indices $W(T_n)$ and $MS_1(T_n)$.
\end{remark}

In \cite{RamYalJAMC} it is reported that for a connected graph $G$, $W(G)< MS_1(G)$.  This relation can be strengthened as follows:

\begin{proposition}\label{prop:WMS1W}
Let $G$ be a connected graph with minimum degree $\delta $ and maximum degree $\Delta $. Then
\[2\delta W(G)\le MS_{1} (G)\le 2\Delta W(G),\]
and equalities hold in both sides if and only if $G$ is a regular graph.
\end{proposition}

\begin{proof}
Because for any connected graph $G$, $MS_{1} (G)=\sum_{ uv\in E(G) }\left(\sigma (u)+\sigma (v)\right) =\sum_{ u\in V(G) }d(u)\sigma (u) ,$
and for any vertex $u$ of $G$,  $\delta \leq d(u) \leq \Delta $, we have that

\[2\delta WG)\le \sum_{ uv\in E(G) }\left(\sigma (u)+\sigma (v)\right) =\sum_{ u\in V(G) }d(u)\sigma (u)\le  2\Delta WG).\]
Consequently, if $G$ is an $r$-regular graph, we have $MS_1(G)=2rW(G)$.
\end{proof}

\begin{corollary} Let $T_n$ be an $n$-vertex tree. Then
\[(n-1)(3n-1)\leq MS_{1} (T_{n} )\leq \frac{1}{3} n(n-1)(2n-1),\]
where
\begin{enumerate}[{\rm(i)}]
  \item the right-hand side equality holds  if and only if $T_n$ is the path $P_n$;
  \item the left-hand side equality holds  if and only if $T_n$ is the star $S_n$.
\end{enumerate}

\end{corollary}
\begin{proof}
For an $n$-vertex tree $T_n$ we have $W(S_n)\leq W(T_n)\leq W(P_n ),$
where $W(S_n)=(n-1)^2$  and  $W(P_n)=\dfrac{(n^3-n)}{6}$. Therefore, from Proposition \ref{prop:MS1WSz}, we have the following inequalities:
\[ MS_{1} (T_{n} )\le \frac{4n(n-1)(n+1)}{6} -n(n-1)=\frac{1}{3} n(n-1)(2n-1),\]
with equality if and only if $T_n$ is the path $P_n$, and
\[MS_{1} (T_{n} )\ge 4(n-1)^{2} -n(n-1)=(n-1)(3n-1),\]
with equality if and only if $T_n$ is the star $S_n$.
\end{proof}

The following is a direct consequence of Proposition \ref{prop:WMS1W}.
\begin{corollary}
If $G_{be}$ is a benzenoid graph with $\Delta =3$ and $\delta =2$, then
\[4W(G_{be} )\le MS_{1} (G_{be} )\le 6W(G_{be} ).\]
\end{corollary}


It is easy to show that the inequality represented by
\[MS_{2} (G)=\sum_{ uv\in E(G) }\sigma (u)\sigma (v) \le \frac{1}{2}  MSD^{\frac{3}{2}}(G),\]
can be sharpened in the following form:


\begin{proposition}\label{propMS2irr}
Let $G$ be a connected graph with $m$ edges. Then
\[MS_{2} (G)\le \frac 1{2} MSD^{\frac{3}{2}}(G)- \frac {1}{2m}\irr_{\rm Tr}(G)^2,\]
with equality if and only if $\left|\sigma (u)-\sigma (v)\right|$ is constant for any $uv\in E(G)$.
\end{proposition}

\begin{proof}
Using Cauchy-Schwartz inequality we have
\begin{align*}
\left\{\frac{1}{m} \sum_{ uv\in E(G) }\left|\sigma (u)-\sigma (v)\right| \right\}^{2} &\le \frac{1}{m} \sum_{ uv\in E(G) }\left(\sigma (u)-\sigma (v)\right)^{2},\\
& =\frac{1}{m} \sum_{ uv\in E(G) }(\sigma^{2} (u)+\sigma^{2} (v))-\frac{2}{m}\sum_{ uv\in E(G)}\sigma (u)\sigma (v),
\end{align*}
with equality if and only if $\left|\sigma (u)-\sigma (v)\right|$ is constant for any $uv\in E(G)$. It follows that
\[MS_{2} (G)\le \frac 1{2} MSD^{\frac{3}{2}}(G)- \frac {1}{2m}\irr_{\rm Tr}(G)^2,\]
with equality if and only if $\left|\sigma (u)-\sigma (v)\right|$ is constant for any $uv\in E(G)$.
\end{proof}

\begin{corollary}
Let $G$ be a connected graph with $m$ edges. If $\diam(G)\leq 2$, then
\[MS_{2} (G)\leq \frac 1{2} MSD^{\frac{3}{2}}(G)- \frac {1}{2m}\irr(G)^2,\]
with equality if and only if $\left|d(u)-d(v)\right|$ is constant for any $uv\in E(G)$.
\end{corollary}
\begin{proof}
Let $G$ be a connected graph with $m$ edges. It follows from Remark \ref{remark:eccent} that for any $uv\in E(G)$, $\left|d(u)-d(v)\right|$ is constant if $\diam(G)\leq 2$. Now the result follows from Lemma \ref{lem:diam2} and Proposition \ref{propMS2irr}.
\end{proof}

\begin{lemma}[\cite{Tom-16},\cite{Doy-Grav}]\label{lem:WLo}
Let $G$ be a connected graph with $n$ vertices and $m$ edges. Then
\[W(G)\ge n(n-1)-m,\]
with equality if and only if $\diam(G)\leq 2$.  (For example, the equality holds for complete graphs, complete bipartite and complete multipartite graphs, moreover wheel graphs and windmill graphs composed of triangles.)
\end{lemma}
\begin{proposition}\label{prop:ktansmision}
Let $G$ be a connected $k$-transmission regular with $n$ vertices and $m$ edges. Then
\[k=\dfrac{2W(G)}{n}\ge 2(n-1)-\frac{2m}{n}, \]
with equality if and only if $\diam(G)\leq 2$.
\end{proposition}
\begin{proof}
Since $G$ is $k$-transmission regular, $W(G)=\frac{2k}{n}$. Now the result follows from Lemma \ref{lem:WLo}.
\end{proof}

\begin{proposition}
Let $G$ be a connected graph with $n$ vertices and $m$ edges. Then
\[ MS^1(G) \ge 4(n-1)W(G)-MS_{1} (G)\ge 4(n-1)\left(n^{2} -n-m\right)-MS_{1} (G),\]
and equalities hold in both sides simultaneously if  $\diam(G)\leq 2$.
\end{proposition}

\begin{proof}
The result follows directly, using Lemma \ref{lem:WLo} and
Proposition \ref{prop:MS1lower}.
\end{proof}

\begin{proposition}
Let $G$ be a connected graph with $n$ vertices and $m$ edges. Then
\begin{equation}\label{eq:MS1MS2}
 MS_{1} (G)\le \sqrt{m\left\{ MSD^{\frac{3}{2}}(G)  +2MS_{2} (G)\right\}},
\end{equation}
with equality if and only if $\sigma(u)+\sigma(v)$ is constant for each edge $uv\in E(G)$.
\end{proposition}

\begin{proof}
Using the Cauchy-Schwartz inequality, we obtain
\begin{align*}
\left\{\frac{1}{m} \sum _{uv\in E(G)}\left(\sigma (u)+\sigma (v)\right) \right\}^{2} & \le \frac{1}{m} \sum _{uv\in E(G)}\left(\sigma (u)+\sigma (v)\right)^{2}\\
&=\frac{1}{m}  \left\{\sum _{uv\in E(G)}\left(\sigma^{2} (u)+\sigma^{2} (v)\right)+2\sum _{uv\in E(G)}\sigma (u)\sigma (v)  \right\},
\end{align*}
with equality if and only if $\sigma(u)+\sigma(v)$ is constant for each edge $uv\in E(G)$.
This implies that
\[\left\{\frac{1}{m} MS_{1} (G)\right\}^{2} \le \frac{1}{m}
\left\{MSD^{\frac{3}{2}}(G) +2MS_{2} (G)\right\},\]
with equality if and only if $\sigma(u)+\sigma(v)$ is constant for each edge $uv\in E(G)$.
Consequently, we have

\[MS_{1} (G)\le \sqrt{m\left\{MSD^{\frac{3}{2}}(G)+2MS_{2} (G)\right\}}.  \]
\end{proof}

Let $G$ be a connected graph with $n$ vertices. Let us define the topological invariant $\Phi(G)$ as follows
$$\Phi (G) = \dfrac{{{{\left( {\sum\limits_{u \in V(G)} {\sigma (u)} } \right)}^2}}}{{n\sum\limits_{u \in V(G)} {{\sigma ^2}(u)} }} = \frac{{4{W}(G)^2}}{{nMS^1(G)}}.$$
The following theorem shows that $\Phi (G)$ quantify the degree of transmission regularity of a connected graph $G$.

\begin{theorem}\label{thm:phiirreg}
Let $ G $ be a connected graph with $n$ vertices. Then
$\Phi (G)\leq 1,$
  with equality if and only if $G$ is transmission regular.
\end{theorem}

\begin{proof}
Using Cauchy-Schwartz inequality, we obtain
\begin{align*}\label{}
\left\{ \sum _{u\in V(G)}\sigma (u) \right\}^{2} & \le
n\sum _{u\in V(G) }\sigma (u)^{2},
\end{align*}
with equality if and only if $\sigma(u)=\sigma(v)$ for each $u,v\in V(G)$.
This completes the proof.
\end{proof}

\begin{proposition}\label{prop:Disradious}
Let $G$ be a connected  graph with $n$ vertices and $m$-edges. If $\rho _D(G)$ denotes the distance spectral radius of $G$, then
$$ 2(n-1)-\frac{2m}{n} \leq \frac{{2W(G)}}{n} \leq {\rho_D}(G).$$
The left-hand side equality holds if and only if $\diam(G)\leq 2$.
The right-hand side equality holds if and only if $G$ is transmission regular.
\end{proposition}
\begin{proof}
The left-hand side inequality is noting but Lemma \ref{lem:WLo}.
From Theorem \ref{thm:phiirreg} and {\rm \cite[Theorem 5.5]{Aou-Han-14}} one obtains that
$\frac{{2W(G)}}{n} \leqslant \sqrt {\frac{1}{n}MS^1(G)} \leqslant {\rho _D}(G),$
with equality if and only if $G$ is transmission regular.
\end{proof}

Let us finish this section with following result showing how $W(G)$, $MS_1(G)$ and
$\xi^{d}(G)$ relates to each other.

\begin{theorem}[\cite{Il-EDS}]\label{thm:}
Let $ G $ be a connected graph on $ n \geqslant 3 $ vertices. Then
\[MS_1(G)\leqslant 2n W(G)-\xi^{d}(G),\]
with equality if and only if $G\cong P_{4}$, or $ G\cong K_{n}-ke $, for $k = 0, 1,\ldots, \lfloor\frac{n}{2}\rfloor $.
\end{theorem}

\section{Vertex and edge transitive graphs}

In this section, following Darafshe \cite{DarafsheHypercube},\cite{Mohamad-darafshKneser}, we aim to use a method which applies group theory to graph theory.
For more details regarding the theory of groups and graph theory one can see \cite{Dixon} and \cite{Godsil}, respectively.

Let $\Gamma$ be a group acting on a set $X$. We shall denote the action of $\alpha\in \Gamma$ on $x\in X$ by $x^{\alpha}$. Then $U\subseteq X$ is call an \emph{orbit} of $\Gamma$ on $X$ if for every $x,y \in U$ there exists  $\alpha\in \Gamma$ such that $x^\alpha=y$. The action of group
 $\Gamma$  on $X$ is called \emph{transitive} if $X$ is itself an orbit of $\Gamma$ on $X$.

Let $G$ be a graph. A bijection  $ \alpha $  on $V(G)$ is called an \textit{automorphism} of $G$ if it preserves $E(G)$. In other words, $\alpha$ is an automorphism if  for each $u,v\in V(G)$, $e =  uv\in E(G)$ if and only if
$u^{\alpha}  v^{\alpha}  \in E(G)$. Let us denote by $Aut(G)$ the set of all automorphisms of $G$.
It is known that $ Aut(G) $  forms a group under the composition of mappings. This is a subgroup of the symmetric group on $V(G)$.
Note that $Aut(G)$ acts on $V(G)$ naturally, i.e., for each $\alpha\in Aut(G)$ and $v\in V(G)$ the action of $\alpha$ on $v$, $v^\alpha$, is defined as $\alpha(v)$.
The action of $Aut(G)$ on $V(G)$ induces an action
on $E(G)$. In fact, for $\alpha\in Aut(G)$ and $e=uv\in E(G)$, the action of $\alpha$ on $e=uv$, $e^\alpha$, is defined as $u^{\alpha} v^{\alpha}$.

A graph $G$ is called \textit{vertex-transitive} (\textit{edge-transitive}) if
the action of $Aut(G)$ on $V(G)$ ($E(G)$) is transitive.

Let $G$ be a graph, $V_1, V_2,\ldots, V_t$ be the orbits of $Aut(G)$
under its natural action on $V (G)$. Then for each $1\leq i\leq t$ and for $u ,v\in V_i$, $\sigma(u)  = \sigma(v) $. In particular, if $G$ is vertex transitive ($t=1$), then for each $u ,v\in V(G)$, $\sigma(u)  = \sigma(v) $. Therefore vertex-transitive graphs are transmission regular.
It is known that any vertex-transitive graph is (vertex degree) regular \cite{Godsil} and transmission regular \cite{DarafsheHypercube}, but note vise versa.

\begin{lemma}\label{lem:tr-regular}
Let $G$ be a connected $k$-transmission regular graph with $n$ vertices and $m$ edges. Then
\[
SJ(G)=\frac{m^2}{(m-n+2)\sqrt{2k}}, \quad
GAS(G)=\frac{n}{2}, \quad HS(G)=\frac m{k},\]
\[J(G)=\frac{m^2}{(m-n+2)k}.\]
\end{lemma}


\begin{lemma}\label{cor:ver-tr-regular-Weiner}
Let $G$ be a connected vertex-transitive graph with $n$ vertices and $m$ edges and the valency $r$. Then

\[
SJ(G)=\frac{m^2\sqrt{n}}{2(m-n+2)\sqrt{W(G)}}, \quad
GAS(G)=\frac{2W(G)}{n},
\]
\[HS(G)=\frac{nm}{2W(G)}=\frac{n^2r}{4W(G)},\]
\[
J(G)=\frac{m^2n}{2(m-n+2)W(G)}=\frac{mn^2r}{4(m-n+2)W(G)}.
\]
\end{lemma}

\begin{proof}
If $G$ is a connected vertex-transitive graph with $n$ vertices and $m$ edges, then $G$ is of valency $r$ ($r$-regular) and $k$-transmission regular, for some natural numbers $r$ and $k$. It follows that $2m=nr$ and $2W(G)=nk$.
\end{proof}
\begin{lemma}
Let $G$ be a connected $k$-transmission regular  with $n$ vertices and $m$ edges. Then
\[HS(G)\leq \frac{m}{2(n-1)-\frac{2m}{n}},\]
with equality if and only if $\diam(G)\leq 2$.
\begin{proof}
  Follows from Proposition \ref{prop:ktansmision} and the fact that for a $k$-transmission regular graph $G$ with $n$ vertices and $m$ edges, $HS(G)=\frac{m}{k}$.
\end{proof}
\end{lemma}


\begin{theorem}
	Let $G$ be a connected graph with $n$ vertices and $m$ edges. Let us denote
the orbits of the action $Aut(G)$ on $E(G)$ by $E_1, E_2,\ldots , E_l$. Suppose that for each
$1\leq i \leq t$, $e_i=u_iv_i$ is a fixed edge in the orbit $E_i$. Then
\[
HS(G)=\sum_{i=1}^{l} \frac{2|E_i|}{\sigma(u_i)+\sigma(v_i)},\quad
SJ(G)=\frac m{m-n+2}\sum_{i=1}^{l} \frac{|E_i|}{\sqrt{\sigma(u_i)+\sigma(v_i)}},\]
\[GAS(G)=\frac{n}{2m}\sum_{i=1}^{l}\frac{|E_i|\sqrt{\sigma(u_i)\sigma(v_i)}}{\sigma(u_i)+\sigma(v_i)},
\quad
\irr_\Tr(G)=\sum_{i=1}^{l}|E_i|\left|\sigma(u_i)-\sigma(v_i)\right|,
\]
\[
MS_1(G)=\sum_{i=1}^{l}|E_i|(\sigma(u_i)+\sigma(v_i)),\quad
MS_2(G)=\sum_{i=1}^{l}|E_i|\sigma(u_i)\sigma(v_i),
\]

\end{theorem}

\begin{corollary}\label{cor:edgetransitive}
Let $G$ be a connected graph with $n$ vertices and $m$ edges. If
$G$ is edge-transitive and $uv$ is a fixed edge of $G$, then

\[HS(G)=\frac{2m}{\sigma(u)+\sigma(v)}, \quad
SJ(G)=\frac{m^2}{(m-n+2)\sqrt{\sigma(u)+\sigma(v)}}.\]
\[GAS(G)=\frac{n}{2}\frac{\sqrt{\sigma(v)\sigma(v)}}{\sigma(u )+\sigma(v)},
\quad MS_2(G)=m\sigma(u )\sigma(v)
\]

\[
 \irr_\Tr(G)=m\left|\sigma(u )-\sigma(v)\right|,
 \quad QS_{e} (G)=\left|\sigma(u )-\sigma(v)\right|
\]
\[
QS_{v,e} (G)=\frac{n}{2} \left\{1+\left|\sigma(u )-\sigma(v)\right|\right\},\quad
MS_1(G)=m(\sigma(u )+\sigma(v))
\]
\end{corollary}

Fullerenes are zero-dimensional nanostructures, discovered experimentally
in 1985 \cite{Kroto}. Fullerenes $C_n$ can be drawn for $n = 20$ and for all even $n\ge 24$.
They have $n$ carbon atoms, $\frac{3n}{2}$ bonds, 12 pentagonal and $\frac{n}{2}-10$ hexagonal
faces. The most important member of the family of fullerenes is $C_{60}$ \cite{Kroto}. The smallest fullerene is $C_{20}$.
It is a well-known fact that among all fullerene graphs only $C_{20}$ and $C_{60}$ (see Figure  \ref{C20}) are vertex-transitive \cite{Ghorbani}. Since for every vertex of $v\in V(C_{20})$, $\sigma(v) = 50$ and
for every $v\in V (C_{60})$, $\sigma(v)= 278$, then

\[
SJ(C_{20})=7.5, \quad
GAS(C_{20})=50,\quad  HS(C_{20})=0.6,\]
\[
J(C_{20})=1.5, \quad
SJ(C_{60})=10.73, \quad
GAS(C_{60})=278,
\]
\[HS(C_{60})=0.32,\quad
J(C_{60})=0.9.
\]

\begin{figure}[H]
  \centering
  \includegraphics[width=5cm]{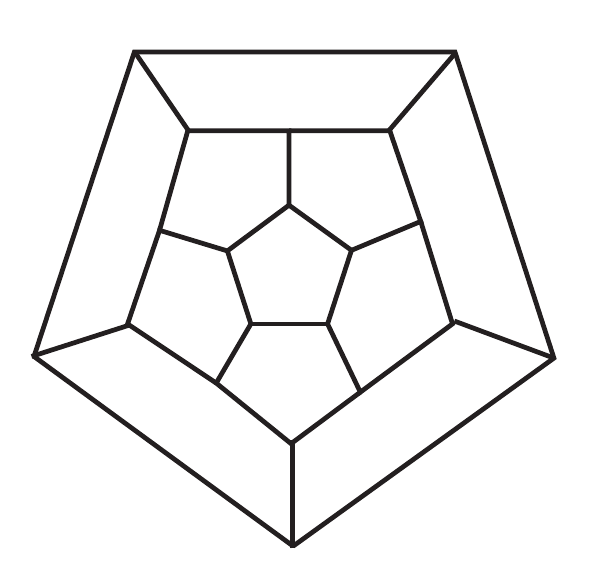}
  \caption{2-dimensional graph of fullerene $C_{20}$}\label{C20}
\end{figure}

A nanostructure called \textit{achiral polyhex nanotorus} (or \textit{toroidal fullerenes} of parameter $p$ and length $q$, denoted by $T=T[p, q]$ is depicted in Figure \ref{fig:apn} and its 2-dimensional molecular graph is in Figure \ref{fig:apnlattice}. It is regular of valency 3 and has $pq$ vertices and $\frac{3pq}{2}$ edges. It follows that

\begin{proposition}\label{prop:Tpq}
\[
SJ(T)=\frac{9{(pq)}^2\sqrt{pq}}{8(pq+2)\sqrt{W(T)}}, \quad
GAS(T)=\frac{2W(T)}{pq},
\]
\[HS(T)=\frac{3{(pq)}^2}{4W(T)},\quad
J(T)=\frac{9{(pq)}^3}{8(pq+2)W(T)}.
\]
\end{proposition}

\begin{figure}[H]
  \centering
  \includegraphics[width=4cm]{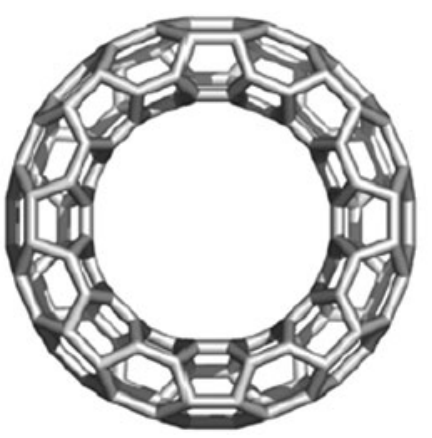}
  \caption{A achiral polyhex nanotorus (or toroidal fullerene) $T[p,q]$}\label{fig:apn}
\end{figure}

\begin{figure}[H]
  \centering
  \includegraphics[width=5cm]{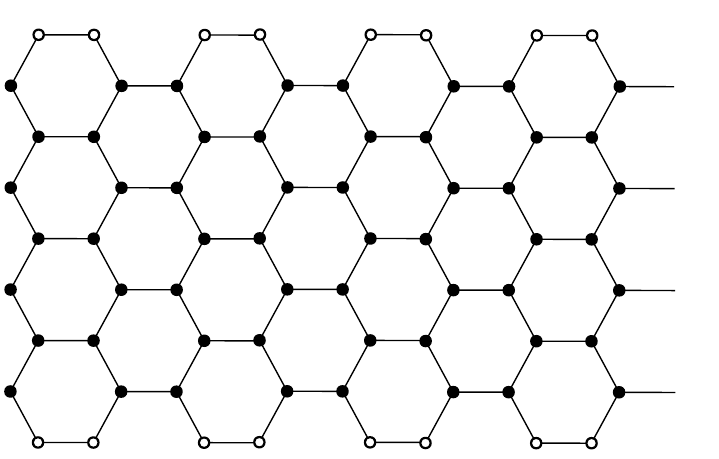}
  \caption{A 2-dimensional lattice for an achiral polyhex nanotorus $T[p,q]$}\label{fig:apnlattice}
\end{figure}

The vertex set of the hypercube $H_n$ consists of all $n$-tuples $(b_1,b_2,\ldots,b_n)$ with
$b_i \in \{0, 1\}$. Two vertices are adjacent if the corresponding tuples differ in precisely
one place. Moreover, $H_n$ has exactly $2n$
vertices and $n2^{n-1}$ edges.


\begin{lemma}[\cite{DarafsheHypercube}]\label{lem:Hn-ver-ed-tr}
The hypercube $H_n$ is $(n2^{n-1})$-transmission regular which is vertex- and edge-transitive.
\end{lemma}

Therefore from Lemma \ref{lem:tr-regular} and Lemma \ref{lem:Hn-ver-ed-tr} we have

\begin{corollary}
\[
SJ(H_n)=\frac{n^22^{2(n-1)}}{(n2^{n-1}-2n+2)\sqrt{n2^{n}}}, \quad
GAS(H_n)=n, \quad HS(H_n)=2n^22^{2({n-1})}.\]
\[J(H_n)=\frac{n^22^{2(n-1)}}{(n2^{n-1}-2n+2){n2^{n-1}}}
\]
\end{corollary}


\end{document}